\documentclass[a4paper,12pt,onecolumn]{article}

\usepackage{hyperref}
\usepackage[vmargin=2cm,hmargin=2cm,headheight=14.5pt,top=2cm,headsep=.5cm]{geometry}
\usepackage[utf8]{inputenc}
\usepackage{bm}
\usepackage{empheq}
\usepackage{stackrel}
\usepackage{cases}
\usepackage{mathtools}
\usepackage{amsthm,amsmath,amscd}
\usepackage{makeidx}
\usepackage[charter]{mathdesign}
\usepackage{tikz-cd}
\tikzcdset{every label/.append style = {font = \small}}

\usepackage{graphicx}
\DeclareMathSizes{12}{12}{8}{6}

\usepackage{cite}
\usepackage{url}
\usepackage[ddmmyy]{datetime}

\usepackage{caption}
\newtheoremstyle{ptheorem}{1em}{0em}{\itshape}{}{\bfseries}{.}{.5em}{\thmname{#1}\thmnumber{ #2}\thmnote{ (\hspace{-.01pt}{#3})}}

\theoremstyle{ptheorem}

\newtheorem{thm}{Theorem}[section]
\newtheorem{pro}[thm]{Proposition}
\newtheorem{lem}[thm]{Lemma}

\newtheoremstyle{hdef}{1em}{0em}{}{}{\bfseries}{.}{.5em}{\thmname{#1}\thmnumber{ #2}\thmnote{ (\hspace{-.01pt}{#3})}}
\theoremstyle{hdef}

\newtheorem{dfn}[thm]{Definition}
\newtheorem{rem}[thm]{Remark}

\newtheoremstyle{obs}{1em}{0em}{\color{blue}}{}{\bfseries}{.}{.5em}{\thmname{#1}\thmnumber{ #2}\thmnote{ (\hspace{-.01pt}{#3})}}
\theoremstyle{obs}

\makeatletter
\newtheoremstyle{premark}{1em}{0em}{
\addtolength{\@totalleftmargin}{1.5em}
\addtolength{\linewidth}{-1.5em}
\parshape 1 1.5em \linewidth}{}{\scshape}{.}{.5em}{}
\makeatother

\theoremstyle{premark}

\numberwithin{equation}{section}
\numberwithin{figure}{section}

\title{A constructive approach towards the Method of Solution-Regions}

\author{
F. Adri\'an F. Tojo\footnote{The author was partially supported by Ministerio de Econom\'ia y Competitividad, Spain, and FEDER, project MTM2013-43014-P, and by the Agencia Estatal de Investigaci\'on (AEI) of Spain under grant MTM2016-75140-P, co-financed by the European Community fund FEDER.} \\
\normalsize e-mail: fernandoadrian.fernandez@usc.es\\
\normalsize \emph{Instituto de Ma\-te\-m\'a\-ti\-cas, Facultade de Matem\'aticas,} \\ \normalsize\emph{Universidade de Santiago de Com\-pos\-te\-la, Spain.}\\ 
}
\date{}

\allowdisplaybreaks
\begin{document}
 \maketitle

\begin{abstract}
In this work we develop the theory of solution-regions with a constructive approach. We also extend the theory to the case of general linear conditions and provide various sets of sufficient hypotheses for existence and multiplicity results.
\end{abstract}

{\small\textbf{Keywords:} System of First Order Differential Equations, Multiple Solutions, Fixed Point Index, General Boundary Conditions, Solution-Regions, Upper and Lower Solutions

\textbf{MSC:} Primary 34B15; secondary 34A34, 34A12}

\section{Introduction}
In a recent paper of Frigon \cite{Frigon} the author developed a theory, named \emph{method of solution-regions}, in order to obtain results concerning the existence and multiplicity of solutions of the equation
\begin{equation}\label{eq1}u'(t)=f(t,u(t))\text{ for a.e. } t\in I:=[a,b],\ u\in B,
\end{equation}
where $I$ is a non-degenerate interval and $B$ is the family of functions satisfying initial conditions of the form
\begin{equation}\label{ci}
u(a)=r
\end{equation}
or the periodic boundary conditions
\begin{equation}\label{cp}
u(a)=u(b).
\end{equation}

The method of solution-regions is an outstanding generalization of various methods of obtaining uniqueness and multiplicity of solutions of differential problems, namely the methods of upper and lower solutions \cite{Pouso2001,Frigon1995}, strict upper and lower solutions \cite{mawhin1987,graefkong,el2015}, solution-tubes \cite{frigon2007,frigono} and strict solution-tubes \cite{frigon2007,frigonlotf}. Furthermore, the method is closely related to that of Gaines and Mahwin concerning what they called \emph{bound sets} \cite{GM1,GM2}. The definition of bound set extends that of solution region in the way presented in Remark \ref{rembs}, but the theory concerning them is developed in a non-comparable way as we will point out.

In this paper we take a constructive approach towards admissible regions. It is in Section 2 that we prove that, indeed, admissible regions, such as are defined in Definition~\ref{defar}, always have an admissible pair, something which formed part of the assumptions before \cite[Definition 3.1]{Frigon}. By unlinking the topological and analytical aspects of admissible regions we are able to reach many interesting conclusions regarding their nature (see Remarks~\ref{br} and~\ref{remint}).

In what concerns solution regions (Section 3) we state the refined definition of \emph{$C$-solution region}. This parameter dependent definition relaxes the restrictions imposed on the admissible pair (condition (H5)) while maintaining a simple proof. We also provide some alternative or complementary hypotheses (see (H4'), (H5') and (H6)) in order to derive, in the next section, existence results.

In Section 4 we generalize the problems studied in~\cite{Frigon} by allowing more general boundary conditions. First we deal with general linear conditions of the kind
\begin{equation}\label{cg}
\Gamma (u-u(a))=r,
\end{equation}
where $r\in{\mathbb R}$ and $\Gamma:{\mathcal C}([a,b],{\mathbb R})\to {\mathbb R}$ is a linear functional (which can be thought as $\Gamma:{\mathcal C}([a,b],{\mathbb R}^n)\to {\mathbb R}^n$ by action on each component) such that $M:=\Gamma(1)\ne0$ ($1$ understood as the constant function $1$ on $[a,b]$). Observe that, by Riesz\-–Mar\-kov\-–Ka\-ku\-ta\-ni Representation Theorem \cite[Theorem 7.2.4]{Ben}, there is a unique regular Borel (signed) measure $\mu$ on $[a,b]$ such that $\Gamma u=\int u\operatorname{d} \mu$ for every $u\in{\mathcal C}([a,b],{\mathbb R})$. Also, condition~\eqref{cg} generalizes condition~\eqref{cp}.
Just take $r=0$ and $\Gamma u=u(b)$ (observe that $\Gamma(1)=1\ne0$).

We will also work with the boundary condition
\begin{equation}\label{cg2}
\Gamma u=r,
\end{equation}
 It is clear that that condition~\eqref{cg2} generalizes condition~\eqref{ci} by defining $\Gamma$ as before. Also, observe that conditions~\eqref{cg} and~\eqref{cg2} overlap in some cases, but neither of them covers all of the cases of the other. At the end of Section 4 we provide an example to which the theory is applied.

 In Section 5 we deal with multiplicity results in the usual way through Fixed Point Index Theory. This scenario requires refined hypotheses (such as (H0'), (H4'') and (H5'')).

 Finally we present our conclusions in Section 6, where we talk about the prowess and limitations of this approach and present some guidelines to overcome the occurring difficulties.

 Throughout this paper we will work with the spaces ${\mathbb R}^n$ with the euclidean norm $\|\cdot\|$ and ${\mathcal C}(X,{\mathbb R}^n)$ the space of continuous functions with the supremum norm $\|\cdot\|_0$ where $X$ is some set. ${\mathbb R}^+$ will denote the interval $(0,+\infty)$.
\section{Admissible Regions}

Let us first state some basic definitions for the theory ahead.
\begin{dfn} Let $A\subset{\mathbb R}\times{\mathbb R}^n$. Given $t\in {\mathbb R}$ and $x\in{\mathbb R}^n$ we write
\[ A_t:=\{x\in{\mathbb R}^n\ :\ (t,x)\in A\},\] 
and
\[ A^x:=\{t\in {\mathbb R}\ :\ (t,x)\in A\}.\] 
If we consider the continuous natural inclusions $i_t:{\mathbb R}^n\to{\mathbb R}\times{\mathbb R}^n$ and $i^x:{\mathbb R}\to{\mathbb R}\times{\mathbb R}^n$ such that $i_t(x)=i^x(t)=(t,x)$, then $A_t=(i_t)^{-1}(A)$ and $A^x=(i^x)^{-1}(A)$. Similarly, we consider the natural projections $\pi_1:{\mathbb R}\times{\mathbb R}^n\to{\mathbb R}$ and $\pi_2:{\mathbb R}\times{\mathbb R}^n\to{\mathbb R}^n$ such that $\pi_1(t,x)=t$ and $\pi_2(t,x)=x$.
\end{dfn}
\begin{dfn}[{\cite[Definition 2.1]{Frigon}}] Let $D\subset I\times{\mathbb R}^n$. A map $f:D\to{\mathbb R}^m$ is a \emph{Carathéodory function} if
\begin{enumerate}
\item $f(t,\cdot)$ is continuous on $D_t$ for almost every $t\in I$,
\item $f(\cdot,x)$ is measurable for all $x\in\pi_2(R)$,
\item for all $k\in{\mathbb R}^+$, there exists $\psi_k\in L^1(I,{\mathbb R})$ such that $\|f(t,x)\|\le\psi_k(t)$ for a.\,e. $t$ and every $x$ such that $\|x\|\le k$ and $(t,x)\in D$, that is, the set
\[ \{t\in I\ :\ \|f(t,x)\|>\psi_k(t)\text{ for some }x\in R_t,\ \|x\|\le k\}\] 
has measure zero.
\end{enumerate}
\end{dfn}
\begin{dfn}\label{defar} Let $n\in{\mathbb N}$. A set $R\subset I\times{\mathbb R}^n$, is called an \emph{admissible region} if $R$ is compact and
\begin{enumerate}
\item[(H0)] $R_t\ne\emptyset$ for every $t\in I$.
\end{enumerate}

\end{dfn}
\begin{dfn}Let $R\subset I\times{\mathbb R}^n$ be an admissible region.
A pair of continuous functions $(h,p)$ where $h:I\times{\mathbb R}^n\to{\mathbb R}$ and $p=(p_1,p_2):I\times{\mathbb R}^n\to I\times{\mathbb R}^n$ is called an \emph{admissible pair associated to $R$} if
\begin{itemize}
\item[(H1)] $R=h^{-1}{((-\infty,0])}$.
\item[(H2)] The map $h$ has partial derivatives at $(t,x)$ for almost every $t$ and every $x$ with $(t,x)\in (I\times {\mathbb R}^n)\backslash R$ and $\frac{\partial h}{\partial t}$ and $\nabla_x h$ are locally Carathéodory maps on $(I\times {\mathbb R}^n)\backslash R$.
\item[(H3)] $p$ is bounded and such that $p(t,x)=(t,x)$ for every $(t,x)\in R$ and 
\begin{equation}\label{trcon}\left\langle\nabla_xh(t,x),p_2(t,x)-x\right\rangle\le0\text{ for a.\,e. $t$ and every $x$ with $(t,x)\in (I\times {\mathbb R}^n)\backslash R$.}\end{equation}
\end{itemize}
\end{dfn}
\begin{rem}\label{br} There are several things to take into account in this definition.
\begin{enumerate}
\item The definition of admissible region in \cite{Frigon} is more stringent in condition~\eqref{trcon}, where the inequality is taken in the strict sense. We will show that this is not necessary in order to get the same results, provided that we add another condition, (H6), later on, or that we alter condition (H4) ahead (see Remark \ref{remalt4}).
\item The definition of admissible region in \cite{Frigon} imposes the existence of an associated admissible pair (that is, it includes axioms (H1)-(H3) in the definition), something which we prove is always the case.
\item The definition of admissible region in \cite{Frigon} does not require explicitly $R$ to be compact, although that is a direct consequence of it.
\item How do we interpret the set on which condition~\eqref{trcon} holds? In the proof of \cite[Theorem~5.1]{Frigon} they use that, \emph{for any function $u\in W^{1,1}(I,{\mathbb R}^n)$ such that $(t,u(t))\in R$ and for every $t\in I$ almost everywhere on $\{t\ :\ h(t,u(t))>0\}$,
\[ \left\langle\nabla_xh(t,u(t)),p_2(t,u(t))-u(t)\right\rangle<0.\] }
In order to ensure this inequality holds in terms of a variable $x$ instead of $u(t)$ it is enough to have that the set
\[ \{t\in I\ :\ \left\langle\nabla_xh(t,x),p_2(t,x)-x\right\rangle\ge0\text{ for some }x\in R_t\}\] 
has measure zero.
\item In \cite{Frigon} the author highlighted, by providing several examples, the versatility of admissible regions, showing that they may have corners, their boundary may not be smooth and that they may not be proximate retracts. With the definition presented here it is obvious that those are possibilities for admissible regions.
\item If $(h,p)$ is an admissible pair and $\beta\in{\mathcal C}^1(I,{\mathbb R}^+)$, then $(\widetilde h,p)$ is also an admissible pair where $\widetilde h(t,x)=\beta(t)h(t,x)$. Hence, if $h(t,R_t)$ is bounded for every $t\in I$, taking $\beta$ such that $\beta(t)<1/\sup \|h(t,R_t)\|$, we have that $\widetilde h$, defined as before, is bounded.
\item If $(h_k,p)$ are an admissible pairs for $k=1,2$; then $(h_1+h_2,p)$ is an admissible pair.
\end{enumerate}
\end{rem}

Now we are ready to state some definitions and results in order to construct an admissible pair for a given region function.
\begin{dfn}Let $C,D\subset{\mathbb R}^n$. We define the \emph{distance of a point $x$ to the set} $C$ as
\[ d_C(x):=\inf\{||x-y||\ :\ y\in C\}.\] 
Analogously, we define the distance between $C$ and $D$ as
\[ d(C,D):=\inf\{||x-y||\ :\ x\in C,\ y\in D\}.\] 
Furthermore, if $C$ is convex, we denote by $P_C:{\mathbb R}^n\to C$ the \emph{projection} onto $C$. Remember that $P_C$ is a continuous function \cite[p. 108]{Holmes}.
\end{dfn}
\begin{thm}[{\cite[p. 63]{Holmes}}]\label{thmci} Let $C$ be a closed convex subset of a Hilbert space 
$H$, and let $\varphi(x): =\frac{1}{2} d_C(x)^2$ for every $x\in H$. Then $\varphi$ is a $C^1$ convex function on $H$ and $\nabla \varphi(x)=x-P_C(x)$.
\end{thm}

\begin{pro}[{\cite[Proposition 3.3.6]{Krantz}}]\label{prowhit} Let $R\subset{\mathbb R}^n$ be a closed set. Then there is\\$\varphi\in{\mathcal C}^\infty({\mathbb R}^n,[0,+\infty))$ such that $\varphi^{-1}(0)=R$.
\end{pro}
\begin{rem}\label{remmin}Observe that, in Proposition~\ref{prowhit}, since $\varphi$ attains an absolute minimum at every point of $R$, $\nabla \varphi(x)=0$ for every $x\in R$.
\end{rem}

\begin{thm}\label{thmar} Let $R\subset I\times{\mathbb R}^n$ be an admissible region. Then $R$ has an associated admissible pair $(h,p)$.
\end{thm}
\begin{proof} Let $r\in{\mathbb R}^+$ such that $\pi_2(R)\subset C:=B_{{\mathbb R}^n}[0,r]$. Let $D:=B_{{\mathbb R}^n}(0,r+1)$ and $\varphi:{\mathbb R}^{n+1}\to{\mathbb R}$ be the function provided by Proposition~\ref{prowhit} for $R\cup(I\times({\mathbb R}^n\backslash D))$. Define
\[ h(t,x):=\begin{dcases}\varphi(t,x), & (t, x)\in I\times C,\\ [1- d_C(x)^2]\varphi(t,x)+\frac{1}{2} d_C(x)^2, & (t, x)\in I\times ( D\backslash C),\\ \frac{1}{2} d_C(x)^2, & (t, x)\in I\times ({\mathbb R}^n\backslash D).\end{dcases}\] 
 $h$ is of class ${\mathcal C}^1$ and $R=h^{-1}((-\infty,0])$, so (H1) and (H2) are satisfied. In fact, 
 \[ \nabla_xh(t,x):=\begin{dcases}\nabla_x\varphi(t,x), & (t, x)\in I\times C,\\ -2[x-P_C(x)]\varphi(t,x)+[1-d_C(x)^2]\nabla_x\varphi(t,x)+x-P_C(x), & (t, x)\in I\times ( D\backslash C),\\ x-P_C(x), & (t, x)\in I\times ({\mathbb R}^n\backslash D).\end{dcases}\]  
 Furthermore, $\nabla_x h(t,x)=0$ for every $(t,x)\in R$ (see Remark~\ref{remmin}).
Let $p=(p_1,p_2):I\times{\mathbb R}^n\to I\times{\mathbb R}^n$ with $p_1(t,x)=t$ and $p_2(t,x):=x-\nabla_xh(t,x)=$
 \[ \begin{dcases} x-\nabla_x\varphi(t,x), & (t, x)\in I\times C,\\ 2[x-P_C(x)]\varphi(t,x)-[1-d_C(x)^2]\nabla_x\varphi(t,x)+P_C(x), & (t, x)\in I\times ( D\backslash C),\\ P_C(x), & (t, x)\in I\times ({\mathbb R}^n\backslash D).\end{dcases}\]  
$p$ is continuous, bounded, and $p(t,x)=(t,x)$ for every $(t,x)\in R$.
Furthermore,
\[ \left\langle\nabla_xh(t,x),p_2(t,x)- x\right\rangle=-\|\nabla_x h(t,x)\|^2\le0,\] 
so (H3) holds as well.

\end{proof}

\section{Solution Regions}

In the previous section we have presented the basic aspects of admissible regions, including the construction of admissible pairs. Unfortunately, the properties of these regions are independent of the problem of study (in this case problem \eqref{eq1}). It is for this reason that we need the concept of solutions regions.

\begin{dfn} Let $C\in{\mathbb R}^+$. A triple $(R,(h,p))$ where $R\subset I\times{\mathbb R}^n$ is an admissible region and $(h,p)$ is an admissible pair for $R$ is called a \textit{$C$-solution region of problem~\eqref{eq1}}, if the following hold:
\begin{enumerate}
\item[(H4)] $\frac{\partial h}{\partial t}(t,x)+\left\langle\nabla_xh(t,x),f(p(t,x))\right\rangle\le0$ for a.e. $t$ and every $x$ with $(t,x)\not\in R$.
\item[(H5)] $h(a,u(a))\le0$ or $h(a,u(a))\le h(b,u(b))$ for any $u\in B$ such that $\|u\|_0\le C$.
\end{enumerate}
\end{dfn}
\begin{rem} In \cite{Frigon} they talk about solution regions and not $C$-solution regions. We will use this extra information given by the parameter $C$ for the more general boundary conditions~\eqref{cg} and~\eqref{cg2} we deal with in this work.
\end{rem}
\begin{rem}\label{remint}We can rewrite condition (H4) as
$\left\langle\nabla h(t,x),\widetilde f(p(t,x))\right\rangle\le0$ for a.e. $t$ and every $x$ with $(t,x)\not\in R$ where $\widetilde f=(1,f)$. Written in this way, this expression is reminiscent of the \emph{transversality condition} \cite[Equation (2.1)]{Cid}, necessary for some generalized Lipschitz uniqueness results, where $\widetilde f$ is the function defined to transform a non-autonomous problem into an autonomous one.

Thus expressed, (H4) has a geometric interpretation similar to the transversality condition. The function $\nabla h(t,x)$ can be thought as a normal field to the local foliation in $I\times {\mathbb R}^n$ where the leaves are defined by $h^{-1}(c)$ for each value $c\in{\mathbb R}$. Hence, (H4) implies that the function $\widetilde f$ points, at each point in $h^{-1}(c)$, in the direction of decreasing value of $c$, that is, against the flow $\nabla h(t,x)$ that takes one leave to another.
\end{rem}

\begin{rem}
In this article we will only use the statement of (H5) in the case conditions $B$ refer to conditions~\eqref{cg} or~\eqref{cg2}. The reader can check that, for the cases~\eqref{ci} and~\eqref{cp} (studied in \cite{Frigon}), condition (H5) can be rewritten as
\begin{itemize}
\item $h(a,r)\le0$ or $h(a,u(a))\le h(b,x)$ for any $x\in{\mathbb R}^n$, $\|x\|\le C$ in the case of condition~\eqref{ci} and
\item $h(a,x)\le0$ or $h(a,x)\le h(b,x)$ for any $x\in{\mathbb R}^n$, $\|x\|\le C$.
\end{itemize}
They are used with a similar statement in \cite{Frigon}.
\end{rem}
\begin{rem}\label{rembs} The concept of solution region is closely related to that of bound set in \cite{GM1,GM2}.
    \begin{dfn}[\cite{GM2}] We say that a set $A\subset I\times{\mathbb R}^n$ is a \emph{bound set relative to problem~\eqref{eq1}} if
        \begin{enumerate}
            \item $A$ is open in the relative topology of $I\times{\mathbb R}^n$.
            \item For any $(t_0,x_0)\in\partial A$ with $t_0\in(a,b)$ there  exists $V\in{\mathcal C}^1(I\times{\mathbb R}^n,{\mathbb R})$ such that
            \begin{enumerate}
                \item $A\subset V^{-1}(-\infty,0)$,
                \item $V(t_0,x_0)=0$,
                \item $\frac{\partial V}{\partial t}(t_0,x_0)+\left\langle\nabla_xV(t_0,x_0),f(t_0,x_0)\right\rangle\ne 0$.                
                \end{enumerate}
            
            \end{enumerate}
        
        \end{dfn}
Let $(R,(h,p))$  be a solution region of   problem~\eqref{eq1}. and consider $A:=\mathring R$. $A$ is open in the relative topology of $I\times {\mathbb R}^n$. Fix $(t_0,x_0)\in \partial A$. Using Proposition \ref{prowhit}, let $\psi\in{\mathcal C}^\infty(I\times{\mathbb R}^n,{\mathbb R}^n)$ such that $\psi\ge 0$ and $\psi^{-1}(0)={\mathbb R}^n\backslash A$. Let $V=h-\psi$. Taking the functional $V$ for every $(t_0,x_0)\in\partial A$ we can show that $A$ is a bound set relative to problem~\eqref{eq1}. Clearly, in the relative topology of $I\times {\mathbb R}^n$, not all compact sets are the closure of an open set and not all open sets have compact closures but, in practice, bound sets generalize solution regions.

The results concerning bound sets in \cite{GM2} have a different flavor than those presented here. For instance they require bound sets to be \emph{autonomous}, that is, with $V$ independent of $t_0$ and $t$ --cf. \cite[Definition 3.2]{GM2} and the $f$ occurring in problem~\eqref{eq1} has to be continuous --see \cite[Theorem 3.1]{GM2}.
    \end{rem}

As said before, in \cite{Frigon}, the inequality in condition~\eqref{trcon} is taken in the strict sense. We will need the following condition to make up for it.
\begin{itemize}
\item[(H6)] Let $(R,(h,p))$ be a solution region of problem~\eqref{eq1}. Then there exists $\widehat h$ such that $(R,(\widehat h,p))$ is a solution region of problem~\eqref{eq1} and there exist $t_1,t_2\in I$ such that $\widehat h(t_1,x)= h(t_1,x)$ for every $x\in{\mathbb R}^n$ and $\widehat h(t_2,x)\ne h(t_2,x)$ for every $x\in\ {\mathbb R}^n\backslash R_{t_2}$.
\end{itemize}

In the light of Remark~\ref{br}, points 6 and 7, condition (H6) is not as stringent as it may seem. In this line, the following lemma shows a sufficient condition for (H6) to hold.
\begin{lem} Let $(R,(h,p))$ be a solution region for~\eqref{eq1}. Assume that $h$ is bounded and $(h,p)$ satisfies
{\rm\begin{enumerate}
\item[\rm{(H4')}] There exist $\varepsilon\in{\mathbb R}^+$, $\delta\in{\mathbb R}^+$ and $t_0\in \mathring I$ such that $t\in[t_0-\delta,t_0+\delta]\subset \mathring I$ and\[ \frac{\partial h}{\partial t}(t,x)+\left\langle\nabla_xh(t,x),f(p(t,x))\right\rangle\le-\varepsilon,\]  for a.e. $t\in[t_0-\delta,t_0+\delta]$ and every $x$ with $(t,x)\not\in R$.
\end{enumerate} }
Then (H6) holds.
\end{lem}
\begin{proof}Let $\beta\in{\mathcal C}^1(I,{\mathbb R}^+)$ such that $\beta|_{I\backslash[t_0-\delta,t_0+\delta]}=1$ and $\beta(t)>1$ for $t\in(t_0-\delta,t_0+\delta)$, $\|\beta'\|_0<\varepsilon/\|h\|_0$ for $t\in I$. Define $\widehat h(t,x)=\beta(t)h(t,x)$ for every $(t,x)\in I\times{\mathbb R}^n$. Clearly, $(\widehat h,p)$ is an admissible pair and $\widehat h(a,x)=h(a,x)$ and $\widehat h(b,x)=h(b,x)$, so $\widehat h$ satisfies (H5). Also $\widehat h(t_0,x)\ne h(t_0,x)$ for every $x\in R_{t_0}$.

Furthermore,

\begin{align*} & \frac{\partial \widehat h}{\partial t}(t,x)+\left\langle\nabla_x\widehat h(t,x),f(p(t,x))\right\rangle\\ = & \beta'(t)h(t,x)+\beta(t)\frac{\partial h}{\partial t}(t,x)+\left\langle\nabla_x(\beta h)(t,x),f(p(t,x))\right\rangle \\= & \beta'(t)h(t,x)+\beta(t)\left[\frac{\partial h}{\partial t}(t,x)+\left\langle\nabla_xh(t,x),f(p(t,x))\right\rangle\right] \\ \le &
\beta'(t)h(t,x)-\beta(t)\varepsilon\le \|\beta'\|_0\|h\|_0-\varepsilon\le0.\end{align*}

\end{proof}

\section{Existence results}
Now we recall the following lemma which will be the key to proving various results ahead. We present it in a slightly modified way, although the proof is the same as in \cite{Frigon}.

\begin{lem}[{\cite[Lemma 2.3]{Frigon}}]\label{lembar} Let $z\in{\mathbb R}$, $w\in{\mathcal C}(I,{\mathbb R})$ and $J=w^{-1}((z,+\infty))$. Assume that
\begin{enumerate}
\item $w\in W^{1,1}_{\operatorname{loc}}(J,{\mathbb R})$,
\item $w'(t)\le0$ for a.\,e. $t\in J$,
\item one of the following conditions holds:
\begin{enumerate}
\item $w(a)\le z$,
\item $w(a)\le w(b)$.
\end{enumerate}
\end{enumerate}
Then $w(t)\le z$ for all $t\in I$ or there exists $k>0$ such that $w(t)=k$ for all $t\in I$.
\end{lem}

\begin{rem}\label{remlembar} In Lemma~\ref{lembar}, if $w'(t)<0$ for a.\,e. $t\in J$, then $w$ cannot be constant so, in that case, $w(t)\le z$.
\end{rem}

Let $(h,p)$ be an admissible pair associated to an admissible region $R$. Consider again $r\in{\mathbb R}^n$ and a positive linear functional $\Gamma:{\mathcal C}(I,{\mathbb R})\to {\mathbb R}$. Let $M:=\Gamma(1)$. In what follows we will assume $M>0$. In the case we would like to work with a functional $\Gamma$ such that $M<0$, it is enough to do the change of variables $v=-u$ and study the problem $v'(t)=\widetilde f(t,v(t))$, $\widetilde\Gamma (v-v(a))=r$ ( $\widetilde\Gamma v=r$ respectively), where $\widetilde f(t,x)=-f(t,-x)$ and $\widetilde\Gamma=-\Gamma$.

 Define $g(s):=\Gamma(\chi_{[a,\cdot]}(s))$ for every $s\in I$ where $\chi_{[a,t]}(s)$ denotes the characteristic function of of the interval $[a,t]$. Observe that $g$ is nonnegative. Let $c\in L^1(I,{\mathbb R})$ be such that
$c(t)>\|f(p(t,x))\|$ for a.\,e. $t\in I$ and every $x\in {\mathbb R}^n$. Now, define
\[ f_R(t,x):=\begin{dcases} f(t,x), & (t,x)\in R,\\
f(p(t,x))+c(t)(p_2(t,x)- x), & (t,x)\in (I\times{\mathbb R}^n)\backslash R,\end{dcases}\] 

 and 
\[ \Theta u:= \Gamma\int_a^tf_R(s,u(s))\operatorname{d} s.\] 
Since $\Gamma u=\int u\operatorname{d} \mu$ for some unique regular Borel measure $\mu$ on $I$, we can apply Fubini's Theorem \cite[Theorem 3.7.5]{Ben}, and so
\[ \Theta u= \Gamma\int_a^bf_R(s,u(s))\chi_{[a,t]}(s)\operatorname{d} s=\int_a^{b} f_R(s,u(s))g(s)\operatorname{d} s .\] 

 Let $m:=\max\{\|x\|\ :\ x\in\pi_2(R)\}$, $C:=\max\left\{m,1+\|p_2\|_0\right\}$, $K>C$, $D:=B_{{\mathbb R}^n}[0,C]$, $E:=B_{{\mathbb R}^n}(0,K)$ and ${\mathcal U}:=\{u\in{\mathcal C}(I,{\mathbb R}^n)\ :\ \|u\|_0<K\}$.

 With the previous ingredients we can consider the problem
\begin{align}\label{eql} u'(t)= &\lambda f_R(t,u(t)),\\ \label{mbc}
u(a)= & P_D\left( \frac{\Gamma u- r}{M}\right) .
\end{align}

Observe that condition~\eqref{mbc} is equivalent to $\Gamma(u-u(a))= r$ when $u\in D$. The reason for the projection is that, due to the nature of the boundary condition~\eqref{cg}, we cannot, in general, bound $\|u(a)\|$ (unlike in the cases of~\eqref{ci} and~\eqref{cp}, cf. \cite{Frigon}) which is necessary in order for condition 3.(a) in Lemma~\ref{lembar} to be satisfied. A similar approach using a projection in the case of upper and lower solutions was used in \cite[Section~3]{CaPo}, where the nonlinear condition $L_1(u(a),u(b),u'(a),u'(b),u)=0$ occurred.

We could have presented ondition~\eqref{mbc} as
\begin{equation}\label{mbc2}Mu(a)= P_{MD}\left( \Gamma u- r\right) ,\end{equation}
where $MD:=\{Mx\in {\mathbb R}^n\ :\ x\in D\}$ thanks to the following lemma. Observe that, although~\eqref{mbc2} would allow to write condition~\eqref{mbc}
for the case $M=0$, doing this would be fruitless: for $M=0$ condition~\eqref{mbc2} reads $0=0$. \begin{lem} For any $\alpha\in[0,+\infty)$, $x\in{\mathbb R}^n$, we have that $\alpha P_D(x)= P_{\alpha D}\left( \alpha x\right) $.
\end{lem}
\begin{proof} We have that
\[ \alpha P_D(x)=\begin{dcases} \alpha x, & \|x\|\le C, \\ \alpha C\frac{x}{\|x\|}, & \|\alpha x\|
> C,\end{dcases}\quad P_{\alpha D}\left( \alpha x\right) =\begin{dcases} \alpha x, & \|\alpha x\|\le \alpha C, \\ \alpha C\frac{\alpha x}{\|\alpha x\|}, & \|\alpha x\|> \alpha C.\end{dcases}\] 
Written this way, the result is clear.
\end{proof}

Let, $\displaystyle\Phi u:=\Gamma u- P_{MD}\left( \Gamma u- r\right) $ and consider the operator ${\mathcal J}:[0,1]\times{\mathcal C}(I,{\mathbb R}^n)\to{\mathcal C}(I,{\mathbb R}^n)$ defined by
\[ {\mathcal J}(\lambda,u):=u(a)+\lambda \int_a^tf_R(s,u(s))\operatorname{d} s-(1+\lambda(t-a))(\lambda\Theta u- \Phi u).\] 

Inspired in \cite{Frigon}, we present the following results concerning the existence of solutions.
\begin{pro}\label{proM0} Let $f:I\times{\mathbb R}^n\to{\mathbb R}^n$ be a Carathéodory function and $(h,p)$ an admissible pair associated to an admissible region $R$. Assume (H6) holds and $M>0$. Then, for every $\lambda\in[0,1]$, problem~\eqref{eql},\eqref{mbc} has at least one solution. Moreover, $\operatorname{ind}({\mathcal J}(\lambda,\cdot),{\mathcal U})=(-1)^n$ for every $\lambda\in[0,1]$.
\end{pro}
\begin{proof}${\mathcal J}$ is a continuous and completely continuous operator (cf. \cite[Lemma 2.2]{Frigon}). Also, the fixed points of ${\mathcal I}$ are solutions of~\eqref{eql},\eqref{mbc}. Indeed, if $u={\mathcal J}(\lambda,u)$ then, for every $t\in I$,
\begin{equation*}u(t)=u(a)+\lambda \int_a^tf_R(s,u(s))\operatorname{d} s-(1+\lambda(t-a))(\lambda\Theta u- \Phi u).\end{equation*}
In particular, for $t=a$, we get that $\lambda\Theta u=\Phi u$, so
\begin{equation}\label{eqpfb}u(t)=u(a)+\lambda \int_a^tf_R(s,u(s))\operatorname{d} s.\end{equation}
Applying $\Gamma$ to both sides in~\eqref{eqpfb}, we obtain $\Gamma u=Mu(a)+ \lambda\Theta u=Mu(a)+ \Phi u$. Hence,
\[ u(a)=\frac{\Gamma u-\Phi u}{M}=P_D\left( \frac{\Gamma u- r}{M}\right) ,\] 
and condition~\eqref{mbc} is satisfied.
Also, for a.\,e. $t\in I$,
\[ u'(t)=\lambda f_R(t,u(t)),\] 
so $u$ is a solution of~\eqref{eql},\eqref{mbc} for every $\lambda\in[0,1]$.

Now we prove that the solutions of~\eqref{eql},\eqref{mbc} are bounded by $C$. Assume $u$ is a solution of equation~\eqref{eql},\eqref{mbc} for some $\lambda\in[0,1]$, that is,
\begin{equation*}u(t)=u(a)+\lambda \int_a^tf_R(s,u(s))\operatorname{d} s.\end{equation*}
 Consider first the case $\lambda=0$. In that case $u\equiv x$ for some constant $x\in{\mathbb R}^n$, so
 \[ \|u\|_0=\|x\|=\|u(a)\|=\left\| P_D\left( \frac{\Gamma u- r}{M}\right) \right\|\le C.\] 

Now we study the case $\lambda>0$.

 Then, almost everywhere on $\{t\in I\ :\ \|u(t)\|>C\}$,
\begin{align*}\|u(t)\|'= & \left\langle\frac{u(t)}{\|u(t)\|},u'(t)\right\rangle=\left\langle\frac{u(t)}{\|u(t)\|},\lambda f_R(s,u(t))\right\rangle \\= & \left\langle\frac{u(t)}{\|u(t)\|},\lambda [f(p(t,u(t)))+c(t)(p_2(t,u(t))- u(t))]\right\rangle
\\= & -\lambda c(t)\|u(t)\|+\left\langle\frac{u(t)}{\|u(t)\|},\lambda [f(p(t,u(t)))+c(t)p_2(t,u(t))]\right\rangle
\\ \le & -\lambda c(t)C+\left|\left\langle\frac{u(t)}{\|u(t)\|},\lambda [f(p(t,u(t)))+c(t)p_2(t,u(t))]\right\rangle\right|
\\ \le & \lambda c(t)\left(1+\|p_2\|_0-C\right)<0,
\end{align*}
 where the last inequality is a consequence of the definition of $c$. Then, by Lemma~\ref{lembar} and Remark~\ref{remlembar}, $\|u\|_0\le C$.

By the homotopy property of the fixed point index, $\operatorname{ind}({\mathcal J}(\lambda,\cdot),{\mathcal U})=\operatorname{ind}({\mathcal J}(0,\cdot),{\mathcal U})$ for every $\lambda\in[0,1]$.

 Observe that \[ {\mathcal J}(0,u)(t)=u(a)+ \Phi u.\] 
 is constant for $t\in I$. By the contraction property of the index (see \cite{Granas}),
 \[ \operatorname{ind}({\mathcal J}(0,\cdot),{\mathcal U})=\operatorname{ind}({\mathcal J}(0,\cdot)|_{E},E).\] 

 Now let ${\mathcal P}(\lambda,x)= 2x+ \lambda\Phi x$ for $x\in E$, $\lambda\in[0,1]$. Observe that ${\mathcal P}(1,\cdot)={\mathcal J}(0,\cdot)$. Clearly, for $\lambda=0$, there is no $x\in\partial E$ such that $x={\mathcal P}(0,x)$. Assume there exists $x\in\partial E$ and $\lambda\in(0,1]$ such that $x={\mathcal P}(\lambda,x)$. We have that
 \begin{align*}0= & {\mathcal P}(\lambda,x)-x= x+\lambda\left( \Gamma x- P_{MD}\left( Mx- r\right) \right) =
 x+\lambda\left( Mx- P_{MD}\left( Mx- r\right) \right) .\end{align*}
Thus,
\begin{align*} & (1+\lambda M)x=\lambda P_{MD}\left( Mx- r\right) .\end{align*}
Taking norms,
\begin{align*} & (1+\lambda M)K=\lambda\left\|P_{MD}\left( Mx- r\right) \right\|\le \lambda M C<\lambda MK ,\end{align*}
which is a contradiction.

Therefore, by the homotopy property of the index,
\begin{align*}\operatorname{ind}({\mathcal J}(\gamma,\cdot),{\mathcal U}) = &\operatorname{ind}({\mathcal J}(0,\cdot),E)= \operatorname{ind}({\mathcal P}(1,\cdot),E)=\operatorname{ind}({\mathcal P}(0,\cdot),E) = \operatorname{ind}(\operatorname{Id},E)=(-1)^n.\end{align*}
Thus, ${\mathcal P}(\lambda,\cdot)$ has a fixed point, and hence problem~\eqref{eql},\eqref{mbc} a solution, for every $\lambda\in[0,1]$.
\end{proof}

\begin{thm}\label{thmprin}Let $f:I\times{\mathbb R}^n\to{\mathbb R}^n$ be a Carathéodory function and $(h,p)$ an admissible pair associated to an admissible region $R$ such that $(R,(h,p))$ is a $C$-solution region of~\eqref{eq1}. Assume (H6) holds and $M> 0$. Then problem~\eqref{eq1},\eqref{cg} has a solution $u\in W^{1,1}(I,{\mathbb R}^n)$ such that $(t,u(t))\in R$ for every $t\in I$.
\end{thm}

\begin{proof} By Proposition~\ref{proM0} there is a solution $u$ of problem~\eqref{eql},\eqref{mbc} for $\lambda=1$ such that $\|u\|_0\le C$. Then, a.\,e. on $\{t\in I\ :\ h(t,u(t))>0\}$, we have that, by (H3) and (H4),
\begin{equation}\label{pasoint}\begin{aligned}\frac{\operatorname{d} h}{\operatorname{d} t}(t,u(t))= & \frac{\partial h}{\partial t}(t,u(t))+\left\langle\nabla_xh(t,u(t)),u'(t)\right\rangle \\ = & \frac{\partial h}{\partial t}(t,u(t))+\left\langle\nabla_xh(t,u(t)),f(p(t,x))+c(t)(p_2(t,x)-x)\right\rangle\le 0.
\end{aligned}\end{equation}

Hence, by (H5) and Lemma~\ref{lembar}, either $h(t,u(t))\le 0$, and thus $(t,u(t))\in R$. As a consequence, $u$ is a solution of~\eqref{eq1},\eqref{cg}, or $h(t,u(t))=k$ for every $t\in I$.

Assume we are in the last case. Then, by (H6), there exists another admissible pair $(\widehat h,p_2)$ and two points $t_1,t_2\in I$ such that (H5) holds and $\widehat h(t_1,x)= h(t_1,x)$ for every $x\in{\mathbb R}^n$ and $\widehat h(t_2,x)\ne h(t_2,x)$ for every $x\in\ {\mathbb R}^n\backslash R_{t_2}$. Then, either $\widehat h(t,u(t))$ is not constant, and we are done, or it is constant, and hence \[ k=h(t_1,u(t_1))=\widehat h(t_1,u(t_1))=\widehat h(t_2,u(t_2))\ne h(t_2,u(t_2))=k,\]  a contradiction.
\end{proof}

\begin{rem}\label{remalt4}Observe that, in the previous theorem, we could have dispensed with (H6) if the inequality in (H4) were strict for, in that case, the last inequality in~\eqref{pasoint} would be strict (see Remark~\ref{remlembar}).
\end{rem}
\begin{rem}Observe that in the proof of Theorem~\eqref{thmprin} hypothesis (H5) is used, exclusively, on solutions of problem~\eqref{eql},\eqref{mbc} for $\lambda\in[0,1]$. Hence, we could provide the following weaker form of (H5) taking this into account.
\begin{enumerate}
\item[(H5')] $h(a,u(a))\le0$ or $h(a,u(a))\le h(b,u(b))$ for any solution $u$ of problem~\eqref{eql},\eqref{mbc} such that $\|u\|_0\le C$ for any $\lambda\in[0,1]$.
\end{enumerate}
\end{rem}

Observe that Proposition~\eqref{proM0} was established for condition~\eqref{cg}. Now we will develop a result for the case~\eqref{cg2} in an analogous fashion.

Consider the problem
\begin{equation}\label{eql3}\begin{aligned} u'(t)= &\lambda f_R(t,u(t)),\\
u(a)= & \widetilde\Phi u:=P_D\left( \frac{\Gamma u+Mu(a)- r}{M}\right) .
\end{aligned}\end{equation}

 Consider now the operator ${\mathcal K}:[0,1]\times{\mathcal C}(I,{\mathbb R}^n)\to{\mathcal C}(I,{\mathbb R}^n)$ defined by \[ {\mathcal K}(\lambda,u):=M^{-1}(\widetilde\Phi u-\lambda \Theta u)+\lambda \int_a^tf_R(s,u(s))\operatorname{d} s.\] 

Proposition~\ref{proM} has a proof very similar to that of Proposition~\ref{proM0}. We provide a sketch of it.

\begin{pro}\label{proM} Let $f:I\times{\mathbb R}^n\to{\mathbb R}^n$ be a Carathéodory function and $(h,p)$ an admissible pair associated to an admissible region $R$. Assume (H6) holds. Then, for every $\lambda\in[0,1]$, problem~\eqref{eql3} has at least one solution. Moreover, there exists
\[ C>\max\{\|x\|\ :\ x\in\pi_2(R)\}\] 
such that
\[ \operatorname{ind}({\mathcal K}(\lambda,\cdot),{\mathcal U})=(-1)^n\] 
for every $\lambda\in[0,1]$ with ${\mathcal U}:=\{u\in{\mathcal C}(I,{\mathbb R}^n)\ :\ \|u\|<C\}$.
\end{pro}
\begin{proof}${\mathcal K}$ is a continuous and completely continuous operator and the fixed points of ${\mathcal K}$ are solutions of~\eqref{eql3}. Indeed, if $u={\mathcal K}(\lambda,u)$ then, for every $t\in I$,
\begin{equation}\label{eqpf3}u(t)=M^{-1}(\widetilde\Phi u-\lambda \Theta u)+\lambda \int_a^tf_R(s,u(s))\operatorname{d} s.\end{equation}

Applying $\Gamma$ to both sides in~\eqref{eqpf3} we get $\Gamma u= \widetilde\Phi u$. Also, for a.\,e. $t\in I$, $u'(t)=\lambda f_R(t,u(t))$, so $u$ is a solution of~\eqref{eql3} for every $\lambda\in[0,1]$. The proof continues as in Proposition~\ref{proM}.

\end{proof}

Theorem~\ref{thmprin2} is analogous to Theorem~\ref{thmprin}.
\begin{thm}\label{thmprin2}Let $f:I\times{\mathbb R}^n\to{\mathbb R}^n$ be a Carathéodory function and $(h,p)$ an admissible pair associated to an admissible region $R$ such that $(R,(h,p))$ is a solution region of~\eqref{eq1}. Assume (H6) holds and $M>0$. Then problem~\eqref{eq1},\eqref{cg2} has a solution $u\in W^{1,1}(I,{\mathbb R}^n)$ such that $(t,u(t))\in R$ for every $t\in I$.
\end{thm}

\subsection{An example}
Let $I:=[0,1]$ and consider the problem
\begin{equation}\label{eqeje}\begin{aligned} x'(t)= & -2x(t)e^{y(t)}, & \int_{0}^1 x(s)\operatorname{d} s=1,\\ y'(t)= & -y(t)e^{x(t)}, & \int_{0}^1 y(s)\operatorname{d} s= 1.\end{aligned}
\end{equation}
for a.\,e. $t\in I$. Let $R:=B_{{\mathbb R}^3}[0,2]$, $f(t,x,y):=(-2xe^{y},-ye^x)$, $h(t,x,y)=\frac{1}{2}d_{R}^2(t,x,y)$  and 
\[ p(t,x,y):=(t,x,y)-\nabla h(t,x,y)=P_{R}(t,x,y),\]  for $(t,x,y)\in I\times{\mathbb R}^2$. Now we show that $(R,(h,p))$ is a $2$-solution region of~\eqref{eqeje}.

(H1) $R=h^{-1}{((-\infty,0])}$.

(H2) The map $h$ es continuously differentiable (Theorem~\ref{thmci}).

(H3) $p$ is bounded and such that $p(t,x,y)=(t,x,y)$ for every $(t,x,y)\in R$ and
\begin{align*} & \left\langle\pi_2(\nabla h(t,x,y)),\pi_2(p(t,x,y)-(t,x,y))\right\rangle= -\|\pi_2(\nabla h(t,x,y))\| \\ = & 
\begin{dcases}
 0, & (t,x,y)\in R, \\
 -\frac{\left(x^2+y^2\right) \left(\sqrt{t^2+x^2+y^2}-2\right)^2}{t^2+x^2+y^2}, & (t,x,y)\not\in R, \\
\end{dcases}
\end{align*}
 which is non-positive for $(t,x,y)\in (I\times {\mathbb R}^2)\backslash R$.
 
(H4) Since 
\[ f(p(t,x,y))=
\begin{dcases}
 \left( -2 x e^{-y},-ye^{-x} \right) , &  (t,x,y)\in R, \\
 \left( -\frac{4 x e^{-\frac{2 y}{\sqrt{t^2+x^2+y^2}}}}{\sqrt{t^2+x^2+y^2}},-\frac{2 y e^{-\frac{2 x}{\sqrt{t^2+x^2+y^2}}}}{\sqrt{t^2+x^2+y^2}}\right) , &  (t,x,y)\not\in R,\\
\end{dcases}
\] 
we have that
\begin{align*} & \frac{\partial h}{\partial t}(t,x,y)+\left\langle\nabla_{(x,y)}h(t,x,y),f(p(t,x,y))\right\rangle \\ = &  -\frac{e^{-\frac{2 (x+y)}{\sqrt{t^2+x^2+y^2}}} \left(\sqrt{t^2+x^2+y^2}-2\right) }{\left(t^2+x^2+y^2\right)^{3/2}}\\ & \cdot\left(  4 x^2 e^{\frac{2 x}{\sqrt{t^2+x^2+y^2}}} \sqrt{t^2+x^2+y^2}-t e^{\frac{2 (x+y)}{\sqrt{t^2+x^2+y^2}}} \left(t^2+x^2+y^2\right)+2 y^2 e^{\frac{2 y}{\sqrt{t^2+x^2+y^2}}} \sqrt{t^2+x^2+y^2}\right) 
< 0,\end{align*}
 for a.e. $t$ and every $(x,y)$ with $(t,x,y)\not\in R$. This fact is illustrated in Figure \ref{fig:cont}.

 \begin{figure}
 \centering
 \includegraphics[width=0.5\linewidth]{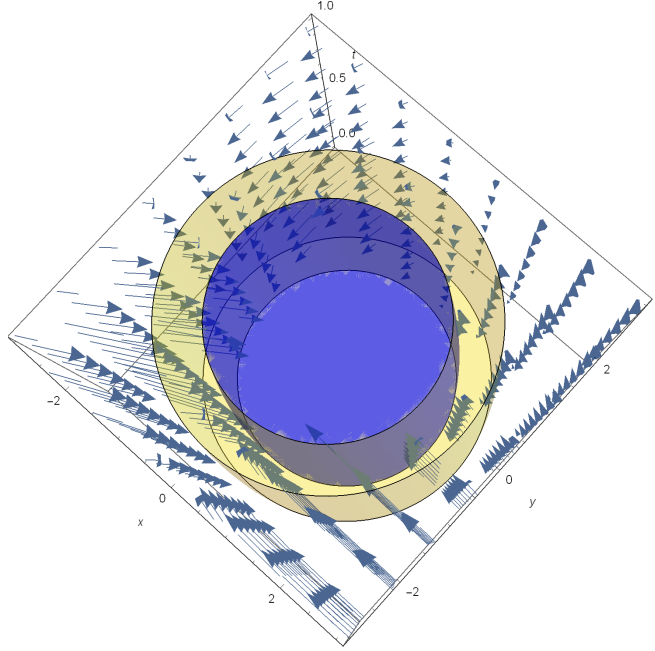}
 \caption{The inner surface is the boundary of the region $R$ (where $h=0$), the outer surface is the level set $h=1/2$ and the vector field the map $\widetilde f=(1,f)$. Observe that, as stated in Remark~\ref{remint}, for $c>0$, $\widetilde f$ points, at each point in $h^{-1}(c)$, in the direction of decreasing value of $c$.}
 \label{fig:cont}
 \end{figure}
 
Observe that $M=1$. Also, we have that $m=\max\{\|x\|\ :\ x\in\pi_2(R)\}=2$, $\|p_2\|_0=2$ and $C:=\max\left\{m,1+\|p_2\|_0\right\}=2$.

(H5') Let us check that, for any solution $u$ of problem~\eqref{eql3} such that $\int_I u= r=\left( 1,1\right) $ and $\|u\|_0\le 2$, we have that $h(0,u(0))\le0$ or $h(0,u(0))\le h(1,u(1))$. In other words, that $d_{R}(0,u(0))\le 0$ or $d_{R}(0,u(0))\le d_{R}(1,u(1))$, which is equivalent to $\|u(0)\|\le 2$ or \[ \max\{\|u(0)\|-2,0\}\le \max\{\sqrt{1+\|u(0)\|^2}-2,0\}.\]

Problem~\eqref{eql3}, in this context, can be expressed as
\begin{equation*}\begin{aligned} u'(t)= &\lambda \begin{dcases} \left( -2 x e^{-y},-ye^{-x} \right) , & (t,x,y)\in R,\\
f(p(t,x,y))+c(t)(p_2(t,x,y)- (x,y)), & (t,x,y)\in (I\times{\mathbb R}^n)\backslash R,\end{dcases}\\
u(-1)= & \begin{dcases} \int_{-1}^0 u+u(-1)- r, & \left\|\int_{-1}^0 u+u(-1)- r\right\|\le 2, \\ 2\frac{\int_{-1}^0 u+u(-1)- r}{\left\|\int_{-1}^0 u+u(-1)- r\right\|}, &  \left\|\int_{-1}^0 u+u(-1)- r\right\|>2.\end{dcases}
\end{aligned}\end{equation*}

Taking norms in the boundary conditions it is clear that  $\|u(-1)\|\le 2$  and, hence, (H5') is satisfied.

We can conclude that there is a solution of problem \eqref{eqeje} in $R=B_{{\mathbb R}^3}[0,2]$. A numerical approximation of this solution is shown in Figure \ref{fig:plot}.

\begin{figure}[h]
\centering
\includegraphics[width=0.7\linewidth]{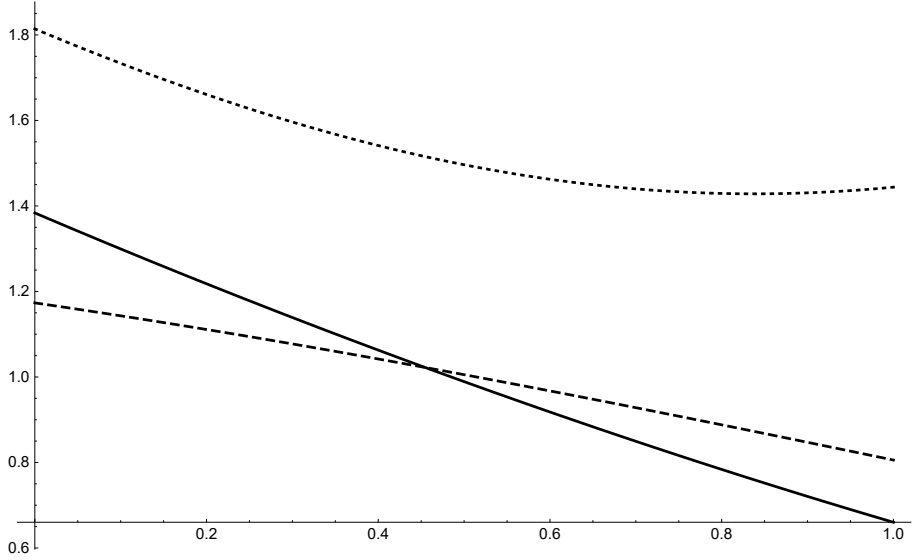}
\caption{Representation of the numerical solution of problem \eqref{eqeje}. The continuous line represents $x(t)$, the dashed line $y(t)$ and the dotted line $\|(t,x(t),u(t))\|$. Observe that the dotted line is always below two.}
\label{fig:plot}
\end{figure}

\section{Multiplicity results}

As we did in the case of admissible regions and solution regions, we present here a slightly modified definition of strict solution region \cite{Frigon}.

\begin{dfn} A solution region $(R,(h,p))$ is a \textit{strict $C$-solution region} of problem~\eqref{eq1} if the following hold:
\begin{enumerate}
\item[(H0')] $\mathring R_t\ne\emptyset$ for every $t\in I$.
\item[(H4'')] There exists $\varepsilon\in{\mathbb R}^+$ such that $\frac{\partial h}{\partial t}(t,x)+\left\langle\nabla_xh(t,x),f(p(t,x))\right\rangle\le0$ for a.e. $t$ and every $x$ with $(t,x)\in h^{-1}((-\varepsilon,0))$.
Also, $\frac{\partial h}{\partial t}$ and $\nabla_x h$ are locally Carathéodory in $h^{-1}((-\varepsilon,0))$.
\item[(H5'')] $h(a,u(a))<0$ or $h(a,u(a))< h(b,u(b))$ for any $u\in B$ such that $\|u\|_0\le C$.
\end{enumerate}
\end{dfn}

\begin{rem} (H0') is a stronger condition than (H0) and (H5'') than (H5). On the other hand, (H4'') complements (H4) by extending it to part of the interior of $R$. Clearly, if $h\ge0$, such as would be the case if it had been obtained by the construction in Theorem~\ref{thmar} or in the previous example, this condition would be vacuous.
\end{rem}

Now we extend \cite[Proposition 6.3]{Frigon} to the case of conditions~\eqref{cg} or~\eqref{cg2}. The proof is basically the same, but we include it for completeness.

\begin{pro}\label{ssl} Let $f:I\times{\mathbb R}^n\to{\mathbb R}^n$ be a Carathéodory function and $(R,(h,p))$ a strict solution region of problem~\eqref{eq1}. Assume $u$ is a solution of problem~\eqref{eq1} such that $(t,u(t))\in R$ for every $t\in I$. Then $(t,u(t))\in \mathring R$ for every $t\in I$.
\end{pro}
\begin{proof} We do the proof in the case of the case of conditions~\eqref{cg} or~\eqref{cg2}. Assume $J:=\{t\in I\ :\ h(t,u(t))=0\}\ne\emptyset$. Observe that $a\not\in J$. If that were the case, $0=h(a,u(a))<h(b,u(b))\le0$, which is a contradiction. Let $r=\min J\in(a,b]$. We have that $h(t,u(t))<0$ for $t\in [0,r)$. Let $\varepsilon\in{\mathbb R}^+$ as in (H4''). Since $u$ and $h$ are continuous, there exist $t_1,t_2\in[0,r)$, $t_1<t_2$ such that $h(t_1,u(t_1))<h(t_2,u(t_2))$ and $h(t,u(t))\in(-\varepsilon,0)$ for all $t\in[t_1,t_2]$. Using (H4'') we get
\begin{align*}0< & h(t_2,u(t_2))-h(t_1,u(t_1))=\int_{t_1}^{t_2}\frac{\operatorname{d} }{\operatorname{d} t}h(t,u(t))\operatorname{d} t\\ = & \int_{t_1}^{t_2}\left( \frac{\partial h}{\partial t}(t,x)+\left\langle\nabla_xh(t,u(t)),f(p(t,u(t)))\right\rangle\right) \operatorname{d} t\le 0,\end{align*}
which is a contradiction and, hence, $J=\emptyset$ and $h(t,u(t))<0$ for every $t\in I$ Since $h$ is continuous and $R=h^{-1}((-\infty,0])$, $\partial R=h^{-1}(0)$, so $(t,u(t))\in \mathring R$ for every $t\in I$.
\end{proof}
\begin{rem} The statement \emph{`$(t,u(t))\in \mathring R$ for every $t\in I$'} in Proposition~\ref{ssl} is slightly stronger than the statement \emph{`$u(t)\in \mathring R_t$ for every $t\in I$'} in \cite[Proposition 6.3]{Frigon}, but the proof in \cite{Frigon} also covers this case. It can be checked that \[ \mathring R\subset\bigcup_{t\in I}(\{t\}\times\mathring R_t),\]  but the reverse content is not in general true. In fact, $\pi_1\left( \bigcup_{t\in I}(\{t\}\times\mathring R_t)\right) $ might be connected while $\pi_1(\mathring R)$ might be not.
\end{rem}

Proposition~\ref{ssl} allows us to prove a variety of multiplicity results in a standard way through Fixed Point Index Theory (cf. \cite{Frigon,infante2016non}).

\section{Conclusions}

As it was pointed out before, we have substituted the original condition
\begin{itemize}
\item[(H3')] $p$ is bounded and such that $p(t,x)=(t,x)$ for every $(t,x)\in R$ and 
\begin{equation*}\left\langle\nabla_xh(t,x),p_2(t,x)-x\right\rangle<0\text{ for a.\,e. $t$ and every $x$ with $(t,x)\in (I\times {\mathbb R}^n)\backslash R$.}\end{equation*}
\end{itemize}
in the definition  of admissible region \cite{Frigon} by the weaker version (H3). As a result we had to ask for (H6) to be satisfied in order to prove Propositions~\ref{proM0} and~\ref{proM} and Theorems~\ref{thmprin} and~\ref{thmprin2}. The reason for this change was the need to prove de existence of admissible pairs associated to a given admissible function. Theorem~\ref{thmar} provides such a result and it may be possible to improve it by obtaining an admissible pair satisfying (H3') instead of (H3).

Such a proof would not be devoid of difficulties. There are several ways to construct a ${\mathcal C}^\infty$ function $f$ that vanishes only on $R$, for instance convolving the indicator function with a bump function, using a Whitney's cover (see \cite[Proposition~3.3.6]{Krantz}, \cite[Section 5.3]{krantz2} or \cite{Whitney}), etc. Furthermore, it seems possible to obtain such a function further satisfying that $(\nabla_xf)^{-1}(\{0\})$ is countable (in the case of a closed subset of ${\mathbb R}$ this is straightforward) although the topology of $R$ may play an important role in the construction. 

Various approaches may be taken at this point. Analytic functions on ${\mathbb R}$ satisfy that, if non-constant, they cannot vanish on a set with accumulation points, a fact that could be relevant when we observe that, Whitney's Extension Theorem \cite[Theorem I]{Whitney} provides an extension that is analytic outside of the original domain of the function.

As said before, the topology of $R$ plays an important role. This is due to the following result.
\begin{lem} Let $f:{\mathbb R}^n\to[0,+\infty)$ be differentiable. Let $R:=f^{-1}(\{0\})=R$. Let $C$ be a bounded connected component of ${\mathbb R}^n\backslash R$. Then there exists $x\in C$ such that $\nabla f(x)=0$.
\end{lem}
\begin{proof} Since $C$ is bounded, $\overline C$ is compact and, hence, being $f$ continuous, it reaches a maximum value in $\overline C$. Since $f>0$ in $C$ and $f$ is $0$ in $\partial C$, the maximum is attained at some point $x\in C$. Since $f$ differentiable in the open set $C$ we have that $f'(x)=0$.
\end{proof}
Furthermore, any open set in ${\mathbb R}^n$ may have at most a countable number of connected components. This is due to the fact that ${\mathbb R}^n$ is a second countable topological space. This means that we have a lower bound for the number of zeros the gradient of a function $f:{\mathbb R}^n\to[0,+\infty)$ may have in $f^{-1}({\mathbb R}^+)$.

The relation of critical points of a function with the connected components of the domain is also consequence of a basic result in Morse Theory known as the Morse inequalities \cite[Theorem~5.2]{Milnor}. To be precise, it relates the topological invariants known as Betti numbers (the first of them being the number of connected components) to the number of the different types of non-degenerate singularities of the function. Thus, Morse Theory might shed light on our problem. We have the following well known results.
\begin{thm}[{\cite[Corollary 2.18]{Matsumoto}}] A non-degenerate critical point is isolated.
\end{thm}
\begin{thm}[{\cite[Theorem 2.20]{Matsumoto} Existence of Morse functions}] Let $M$ be a closed $m$-manifold and $g\in\mathcal C^\infty(\mathbb R^n,\mathbb R)$. Then there exists a Morse function $f\in\mathcal C^\infty(\mathbb R^n,\mathbb R)$ arbitrarily close to $g$.
\end{thm}
Since Morse functions have only non-degenerate critical points, they have at most a countable number of such points, making them another suitable approach to the problem.

\section*{Acknowledgements}
The author would like to acknowledge his gratitude towards Professor Santiago Codesido for his help with the code for the computation of the numerical solution in the example, to Professor Marlène Frigon for her explanatory comments concerning her paper \cite{Frigon} and to Professors Rodrigo López Pouso and Ignacio Márquez Albés for their insightful advice regarding linear functionals defined by Lebesgue-Stieltjes integration.

\end{document}